\documentclass[12pt,leqno]{article}  

\usepackage{graphicx}
\usepackage{amsmath} 
\usepackage{amssymb} 
\usepackage{amsfonts}
\usepackage{amsthm}
\usepackage{enumerate}
\usepackage{makeidx}
\usepackage{verbatim}
\usepackage{epsfig}
\usepackage{hyperref} 
\usepackage{dsfont} 
\usepackage{color}

\newcommand{\change}[1]{{{#1}}}

\makeindex

\usepackage{citeref}   

\usepackage{hyperref} 

\makeindex      
     
\newlength{\fixboxwidth}     
\newcommand{\R} {{\mathbb R}}   
       
\newcommand{\N} {{\mathbb N}}

\newcommand{\e}{\varepsilon }       
\newcommand{\eps}{\varepsilon}

\newcommand\Lip{{\rm Lip}}

\newcommand\lall{\Lambda^{{\rm all}}}     
\newcommand\lstd{\Lambda^{{\rm std}}}     
\newcommand{\wt}{\widetilde }

\newcommand\INT{{\rm INT}}     
\newcommand\APP{{\rm APP}}

\renewcommand{\epsilon}{\varepsilon}

\newcommand{\novakwidebar}[1]{\mbox{\kern1.5pt\hbox{\vbox{\hrule 
height 0.6pt \kern0.35ex
\hbox{\kern-0.15em \ensuremath{#1 }\kern0.0em}}}}\kern-0.1pt}

\theoremstyle{plain}       
     
\newtheorem{thm}{Theorem}   
\newtheorem{cor}{Corollary}      
     
\newtheorem{prop}{Proposition}      
     
\theoremstyle{definition}     
     
\newtheorem{example}{Example}     
\newtheorem{rem}{Remark}

\renewcommand\:{\colon}     

\begin{document}     
     
\title{\bf Algorithms and Complexity for Functions on 
General Domains}

\author{Erich Novak} 
 
\maketitle     



\begin{abstract} 
Error bounds and complexity bounds 
in numerical analysis and infor\-mation-based complexity 
are often proved for functions that are defined 
on very simple domains, such as a cube, a torus, or a sphere.
We study optimal error bounds for the approximation or integration 
of functions 
defined on 
$D_d \subset \R^d$ 
and only assume that $D_d$ 
is a  bounded Lip\-schitz domain. 
Some results are even more general. 
We study three different concepts to measure the complexity:  
order of convergence,
asymptotic constant,
and
explicit uniform bounds, i.e., 
bounds that hold for all $n$ (number of pieces of information) 
and all (normalized) domains. 

It is known for many problems that the order of convergence 
of optimal algorithms does not depend 
on the domain $D_d \subset \R^d$. We 
present examples for which the following statements 
are true: 
\begin{enumerate} 
\item 
Also the asymptotic constant does not depend on the shape of $D_d$
or the imposed boundary values,  
it only depends on the volume of the domain. 
\item
There are explicit and uniform lower (or upper, respectively) 
bounds for the error that are 
only slightly smaller (or larger, respectively) 
than the asymptotic error bound. 
\end{enumerate} 
\end{abstract}

\section{Introduction}

We study optimal error bounds for the approximation or integration 
of functions $f \colon D_d \to \R$, where $D_d \subset \R^d$ 
is a bounded \change{set}.  
We assume that $f \in F(D_d)$ where $F(D_d)$ is a unit ball 
with respect to some norm.
Algorithms $A_n$ may use $n$ function values of $f$, 
this is called standard information 
and denoted by $\lstd$, or $n$ values of general linear functionals, 
this is called general information and denoted by $\lall$. 
We discuss the worst case error of optimal algorithms and use 
common notation such as 
$
e_n(F(D_d), \APP_\infty, \lstd)
$
and
$
e_n(F(D_d), \INT , \lstd)
$
and 
$
e_n(F(D_d), \APP_2, \lall)  
$.
These problems are linear and we know that 
\begin{equation}  \label{eq1}
e_n(F(D_d), \APP_\infty, \lstd) = \inf_{x_1, \dots , x_n \in D_d} 
\sup_{f \in F(D_d), \, f(x_i)=0} \ \Vert f \Vert_\infty 
\end{equation}  
and 
\begin{equation}   \label{eq2} 
e_n(F(D_d), \INT, \lstd) = \inf_{x_1, \dots , x_n \in D_d} 
\sup_{f \in F(D_d), \, f(x_i)=0} \ \int_{D_d} f(x) \, {\rm d} x . 
\end{equation} 
For these problems it is enough to consider linear algorithms. 
Linear algorithms are also optimal for $L_2$-approximation 
if $F(D_d)$ is a unit ball of a Hilbert space and 
in this case 
\begin{equation}  \label{eq3} 
e_n(F(D_d), \APP_2, \lall ) = \inf_{L_1, \dots , L_n} \sup_{f \in F(D_d), 
\, L_i(f) =0} \ \Vert f \Vert_2 = \sigma_{n+1} 
\end{equation} 
coincides with the approximation numbers (linear widths) or singular 
values of the embedding of $F(D_d)$ 
into $L_2$. 
Here the $L_i$ can be arbitrary linear functionals. 
Readers who are not familiar with optimal recovery or information-based 
complexity may read the formulas \eqref{eq1}-\eqref{eq3}  as definitions; 
for more background see \cite{NW08}, in particular Section 4.2.  

\begin{center}
$*$
\end{center}

Much is known about the order of convergence of the 
numbers $e_n$, in particular for simple 
domains $D_d$, 
such as the cube or the torus. 
General bounded Lipschitz domains are studied in 
\cite{DNS1,He09a,He09b,He18,Mi19,NWW04,NT06,Tr07,Tr08,Tr10,Vy06,Vy07b}.
In many cases, the optimal order of the $e_n$ is of the form 
$$
e_n \asymp n^{-\alpha} \, (\log n)^\beta ,  
$$
where $\alpha$ and $\beta$ do not depend on the domain $D_d$. 
It is interesting to know also the exact asymptotic constant 
$$
C: = \lim_{n\to \infty}  e_n \, n^{\alpha} \, (\log n)^{-\beta} ,
$$
if it exists. 
The value of $C$ is known only in rare cases 
(unless $d=1$, we do not discuss the univariate case 
in detail), 
and usually only for very special domains, like the cube, see 
\cite{CKS16,Kr18,KSU14,KSU15}. 
Quite remarkable are recent results of 
Mieth~\cite{Mi19}, to be discussed later, since they hold for 
general \change{bounded open (non-empty) sets} $D_d
\subset \R^d$. 

The order of convergence and the asymptotic constant are not really relevant 
for the applications, where we only can use a ``small'' number of $n$; 
see \cite{NW09} for a drastic example. 
We need explicit (upper and lower) bounds for finite $n \in \N$. 
It is remarkable that some bounds hold uniformly, i.e., 
for all $D_d$ of a given size or volume. 
Explicit lower bounds can be used to prove the curse of dimension
and explicit upper bounds can be used to prove the tractability of certain 
problems. We refer to \cite{NW08,NW10,NW12}, where mainly simple domains,
usually the cube or the torus, are studied.  

\begin{center} 
$*$
\end{center}

We discuss the approximation and integration of 
Lipschitz and H\"older functions in Section~2. 
Using results of Hlawka, Kolmogorov and Tikhomirov,  Sukharev and Chernaya we see 
that, asymptotically, the error behaves like $c \cdot n^{-1/d}$
and $c$ only depends on the volume $\lambda^d(D_d)$ 
of the domain $D_d$. 
We also prove explicit uniform error bounds that hold for every $n \in \N$ 
and every domain $D_d$ (with a given volume). 
\change{For these results we assume that $D_d \subset \R^d$ is a bounded Jordan 
measurable set with an interior point.} 
In Section 3 we study functions with a higher smoothness
and present, in particular, an open problem concerning $C^2$ functions. 

In Section 4 we use the class $\lall$ and mainly present results of 
Mieth concerning $L_2$-approximation of functions from 
the Sobolev spaces $H^r(D_d)$. 
Again the asymptotic constants do not depend on $D_d$ 
and, \change{using} work of Kr\"oger and Li and Yau, 
one can obtain explicit uniform 
bounds that are very close to the asymptotic bounds. 

Along the way, we present several open problems. 

\section{Approximation of H\"older functions} 

\subsection{$L_\infty$-Approximation} 

Assume that $(D, \rho)$ is a bounded metric space and consider the class 
$$
F^\omega (D)  = \{ f\: D \to \R \mid \omega (f,h) \le \omega (h) \} .
$$
Here 
$$
\omega (f,h) = \sup \{ | f(x) - f(y) | \mid \rho(x,y) \le h \}
$$
is the modulus of continuity of $f$ 
and $\omega\: \R^+ \to \R^+$ is assumed to be nondecreasing,
continuous, subadditive with $\lim_{h \to 0} \omega (h) =0$. 
We also need the covering numbers 
$$
c_n = \inf_{x_1, \dots , x_n \in D} \sup_{x \in D} \min_i \rho (x, x_i) .
$$
We start with a result of Sukharev~\cite{Su78}  that can also be found in 
Novak~\cite{No88}.

\begin{prop}
	$$
	e_n (F^\omega (D) , \APP_\infty, \lstd  ) = \omega (c_n) 
	$$
\end{prop}

\change{We often use the following notation. If 
$(a_n)_n$ and $(b_n)_n$ are sequences of (positive) 
real numbers then $a_n \approx b_n$ means that 
$\lim_{n \to \infty} a_n/b_n =1$.
By $a_n \asymp b_n $ we mean that there exist positive 
constants $c$ and $C$ such that 
$$
c \le a_n / b_n \le C 
$$
for all $n$. 
In general we do not have results about the speed of convergence 
(when we write 
$a_n \approx b_n$) or about the size of $c$ and $C$ 
(when we write 
$a_n \asymp b_n $). 
These notions are not meant in a ``uniform'' way, statements containing 
these symbols are not really 
useful in a practical sense.}

\begin{example} 
	Consider the metric $\rho (x,y) = \Vert x-y\Vert_\infty$ on $\R^d$ 
	and subsets $D_d \subset \R^d$. 
	For $D_d = [0,1]^d$ one gets 
	$c_n = \frac{1}{2} m^{-1}$ for 
	\change{$m^d \le n < (m+1)^{d+1}$}  
	and hence, for $\omega (h)=h$, 
	$$
	e_n (F^\omega ([0,1]^d), \APP_\infty, \lstd ) \approx \frac{1}{2} n^{-1/d}.
	$$
	For general bounded sets $D_d$ that contain an interior point we have 
	$$
	e_n( F^\omega (D_d), \APP_\infty, \lstd ) \asymp n^{-1/d} . 
	$$
	For the existence of an asymptotic constant one needs stronger assumptions. 
	If $D_d$ is Jordan measurable with $\lambda^d (D_d) >0$, then 
	$$
e_n (F^\omega(D_d), \APP_\infty, \lstd  ) 
\approx \frac{1}{2} \lambda^d(D_d)^{1/d} \cdot n^{-1/d} .
	$$
	Hence the asymptotic constant only depends on the volume of the domain. 
	Moreover, the explicit uniform lower bound 
	\begin{equation}  \label{lower1} 
	\inf_{D_d} 	e_n (F^\omega(D_d), \APP_\infty, \lstd  ) \cdot 
	\lambda^d(D_d)^{-1/d} = \frac{1}{2} \cdot n^{-1/d} 
	\end{equation} 	
	holds and this uniform lower bound fits nicely to the asymptotic result. 
Upper bounds for the covering numbers  
$c_n$ and hence for the optimal  error bounds 
	$e_n(F^\omega (D_d), \APP_\infty, \lstd )$ are known for particular sets 
	$D_d \subset \R^d$. 

To prove the lower bound \eqref{lower1} it is enough to estimate the volume of all 
$x \in D_d$ with $\min_i \rho(x,x_i) \le \eps$. 
This volume is at most $n$ times the volume of a 
	$\rho$-ball with radius $\eps$ and so we obtain the inequality 
	$$ 
e_n (F^\omega(D_d), \APP_\infty, \lstd  ) 
\cdot \lambda^d(D_d)^{-1/d} \ge  \frac{1}{2} \cdot n^{-1/d} . 
$$	
This inequality is sharp, 
as can be seen if we take $D_d$ as the disjoint union of $n$ 
$\rho$-balls with the same radius. 
\end{example} 

From now on we do not any more consider arbitrary bounded  metric spaces: 
We assume that $D_d \subset \R^d$ is bounded and the metric is induced 
by a norm in $\R^d$. We denote by $B$ the unit 
ball and also write $\Vert \cdot \Vert_B$ for the norm. 
To simplify the formulas we consider 
only Lipschitz functions, hence $\omega (h) = h$. 
We denote the respective space by $F^B(D_d)$; 
it contains all functions $f\: D_d \to \R$  with 
$$
|f(x) - f(y)| \le \Vert x - y \Vert_B.
$$

\begin{thm}  \label{th1} 
Assume that $D_d \subset \R^d$ is a bounded \change{Jordan measurable 
set} with an interior point. 
Then 
\begin{equation} \label{asy1}  
	e_n (F^B(D_d), \APP_\infty, \lstd ) \approx \Theta_B^{1/d}  
	\left( \frac{\lambda^d(D_d)}{\lambda^d(B)}\right)^{1/d} 
	\cdot n^{-1/d},
	\end{equation} 
	where  $\Theta_B$ is the covering constant of $\R^d$ with respect to $B$ and 
	$$
	1 \le \Theta_B \le d \log d + d \log \log d + 5d .
	$$
	Moreover, 
	\begin{equation} 	\label{lower2} 
	\inf_{D_d}   e_n (F^B(D_d), \APP_\infty, \lstd ) \cdot \lambda^d(D_d)^{-1/d}   
	=  \lambda^d(B)^{-1/d} \cdot n^{-1/d}
	\end{equation} 	
\end{thm} 

\begin{proof} 
This is mainly a summary of known results: 
	Hlawka~\cite{Hl49} and 
	Kolmogorov and Tikhomirov~\cite{KT59} proved (independently) sharp results 
	about the covering numbers $c_n$ that yield, together with 
	Proposition 1 of Sukharev~\cite{Su78}, 
	the asymptotic formula~\eqref{asy1}. 
	The bound on $\Theta_B$ is from Rogers~\cite{Ro57}. 	
Only the explicit uniform  lower bound \eqref{lower2} cannot be found in these papers. 
The inequality follows again from a simple volume estimate and the sharpness of the bound follows again 
by the example from above: take $D_d$ as the disjoint union of $n$ balls $\delta B$ with the same radius $\delta$. 
\end{proof} 

\begin{rem} 
Uniform upper bounds do not make sense for this problem since 
	$$
\sup_{D_d} 
	e_n (F^B(D_d), \APP_\infty, \lstd  ) \cdot \lambda^d(D_d)^{-1/d} = \infty  
	$$
	\change{for any fixed $ n \in \N$.} 
	Let us consider the sub-class 
$$
	F^B_0 (D_d) = \{ f \in F^\omega (D_d) \mid f=0 \hbox{ on } \partial D_d \}
$$
of functions that vanish on the boundary of $D_d$. 
	Then similar results hold as in Theorem~\ref{th1},
in particular 	
	$$
	e_n (F^B_0 (D_d), \APP_\infty, \lstd ) \approx \Theta_B^{1/d}  
	\left( \frac{\lambda^d(D_d)}{\lambda^d(B)}\right)^{1/d} 
	\cdot n^{-1/d}.
	$$
	\change{ 
	For the lower bound we apply  \eqref{eq1} and 
	\eqref{asy1}  to the sets $D_d^\eps = \{ x \in D_d 
	\mid d(x, \partial D_d) > \eps \}$. 
	Here it is important that
	$\lim_{\eps \to 0} \lambda^d(D_d^\eps) = \lambda^d(D_d)$, 
	since $D_d$ is Jordan measurable. } 
	Now, instead of~\eqref{lower2}, we obtain a uniform upper bound. 
	It is easy to prove 	
	\begin{equation}  \label{upper1} 
\sup_{D_d}   e_n (F^B_0 (D_d), \APP_\infty, \lstd ) \cdot \lambda^d(D_d)^{-1/d}   
\le 2 \cdot   \lambda^d(B)^{-1/d} \cdot n^{-1/d}  . 
\end{equation} 

\begin{proof}
Assume that the disjoint balls $B_\e (x_i) \subset D_d$ 
with $i=1, \dots , n$ form a complete packing of $D_d$, i.e., 
there is no room for another ball with radius $\e$. 
Then there cannot exist an $x \in D_d$ with a distance of more than 
$2 \e$ from $ \partial D_d \cup \{ x_1, \dots , x_n \}$ and hence 
the radius of information, using the function values at $x_1, \dots , x_n$, 
is at most $2 \e$. 
Then the statement again follows from a simple volume estimate since 
$n \cdot \lambda^d (B_\e) \le \lambda^d (D_d)$. 
 The upper bound \eqref{upper1} is almost optimal since   
	$$
	\lambda^d(B)^{-1/d} \cdot (n+1)^{-1/d} \le 
\sup_{D_d}   e_n (F^B_0 (D_d), \APP_\infty, \lstd ) \cdot \lambda^d(D_d)^{-1/d} . 
	$$
	For this inequality it is enough to consider 
	the case of $n+1$ disjoint balls with the same radius.   
	\end{proof} 
\end{rem} 

\begin{rem} 
Formula~\eqref{asy1} shows that the asymptotic constant 
does not depend on $D_d$, it only depends 
on the volume of the domain. 
Moreover, the asymptotic constant is only by a 
factor $\Theta_B^{1/d}$ larger than the uniform lower bound~\eqref{lower2}. 
This factor is very close to 1, in particular 
if $d$ is large, $\lim_{d\to \infty}  \Theta_B^{1/d} \to 1$. 

	After a suitable normalization, when we put 
	$
	\lambda^d(D_d) = \lambda^d(B) 
$, we even obtain 
$$
	e_n (F^B(D_d), \APP_\infty, \lstd ) \approx \Theta_B^{1/d}  
	\cdot n^{-1/d},
$$
and
$$ 
	\inf_{D_d}   e_n (F^B(D_d), \APP_\infty, \lstd ) =   n^{-1/d} . 
$$
This means that the asymptotic constant very mildly depends on $B$ and 
the explicit uniform lower bound does not depend on $B$ at all. 

For known results on the covering constants $\Theta_B$ see the 
	recent survey~\cite{FT18}.   
\end{rem} 

\begin{rem} 
A special metric is given by the standard norm in $\ell_p^d$, i.e., $B=B_p^d$ 
	is the unit ball in $\ell_p^d$. 
	Then we write 
$\Vert \cdot \Vert_p$ instead of  $\Vert \cdot \Vert_B$ 
	and $F^p(D_d)$ instead of $F^B(D_d)$. It is the space of all functions $f\: D_d \to \R$  with 
$$
|f(x) - f(y)| \le \Vert x - y \Vert_p.
$$
	Of course we can apply Theorem~\ref{th1} in this case and we may use the formula 
$$
	\lambda^d (B_p^d)^{1/d} \approx 2 \Gamma (1+ 1/p) (p e)^{1/p} \cdot d^{-1/p} = c_p \cdot d^{-1/p}  .
$$ 
	Assume again that $D_d \subset \R^d$ is a bounded Jordan \change{measurable set} with an interior point. 
	Then the asymptotic constant is given by 
	\begin{equation} \label{asy2}  
e_n (F^p(D_d), \APP_\infty, \lstd ) \approx \Theta_{B_p^d}^{1/d}  \left( \frac{\lambda^d(D_d)}{\lambda^d(B_p^d)}\right)^{1/d} 
	\cdot n^{-1/d},
	\end{equation} 
	and we obtain from \eqref{lower2}  a uniform lower bound of the form 
	\begin{equation}  \label{lower3} 
	e_n (F^p(D_d), \APP_\infty, \lstd ) \ge  c_p' \,  d^{1/p} \,  \lambda^d(D_d)^{1/d} \cdot  n^{-1/d} . 
	\end{equation} 
This bound heavily depends on the parameter $p$ since
$\lambda^d (B^d_p)^{1/d}$ depends on $p$. 
\end{rem} 

\begin{rem} 
	For large $d$ and $p=2$ and $\lambda (D_d)=1$, 
	formula \eqref{asy1} takes the form 
	$$
	e_n(F^2 (D_d) ), \APP_\infty, \lstd  ) \approx 
	(2 \pi e)^{-1/2} \, d^{1/2} \, n^{-1/d} 
	\approx 0.24197 \, d^{1/2} \, n^{-1/d} .
	$$
	If we take, instead of optimal sample points 
	$x_1, x_2, \dots , x_n$, a regular grid, then the error 
	is $\frac{1}{2} d^{1/2} \, n^{-1/d}$. 
	Hence we only loose a factor of roughly $1/2$ by taking the simplest possible function values 
	instead of the optimal sample points. 

	There is, however, also a different interpretation of the same result: 
	To obtain a given error $\eps$, one needs more than $2^d$ times more function evaluations 
	if one uses a grid instead of optimal function evaluations. 
\end{rem} 

\begin{rem}
	One may use \eqref{lower2} or  \eqref{lower3} to prove the curse of dimension as follows: 
	Assume that the volume $\lambda^d(D_d)$ is one and we consider functions on $D_d$ with 
$$
	|f(x) - f(y)| \le  d^{-1/p}  \Vert x - y \Vert_p.
$$
	Then we need 
	\change{ $n(\e , d) \ge (\tilde c_p / \e)^d$ } 
	function values 
	to reach the error $\eps$, i.e., the problem suffers from the curse of dimension. 
\end{rem} 

\subsection{Integration and $L_1$-approximation} 

Now we study the problem of $L_1$-approximation or numerical integration
and again we assume that $D_d \subset \R^d$ is Jordan measurable with 
$0 < \lambda^d (D_d) < \infty$. 
Asymptotic formulas, even for more general (weighted)  integration problems, 
where proved by Chernaya~\cite{Ch95} and by Gruber~\cite{Gr04}. 
We add an upper bound for the asymptotic constant $\xi_B$  and see that 
it is very close to the lower bound, in particular for large $d$. 
Also the explicit uniform lower bound is new. 

\begin{thm}  \label{th2} 
Assume that $D_d \subset \R^d$ is a bounded \change{Jordan measurable 
set} with an interior point. 
	Then 
	\begin{equation} \label{asy3}  
		e_n (F^B(D_d), \INT, \lstd ) \approx \xi_B  \,  \lambda^d(D_d)   
		\left( \frac{\lambda^d(D_d)}{\lambda^d(B)}\right)^{1/d} 
	\cdot n^{-1/d},
	\end{equation} 
	where  $\xi_B$  is a constant that depends on the norm and 
	$$
	\frac{d}{d+1}  \le \xi_B \le  \frac{d}{d+1}  \Theta_B^{1/d}  \le  \frac{d}{d+1} ( d \log d + d \log \log d + 5d )^{1/d}  .
	$$
	Moreover, 
	\begin{equation} 	\label{lower4} 
		\inf_{D_d}   e_n (F^B(D_d), \INT, \lstd  ) \cdot \lambda^d(D_d)^{-(d+1)/d}   =  \frac{d}{d+1}  \lambda^d(B)^{-1/d} 
		\cdot n^{-1/d} . 
	\end{equation} 	
\end{thm} 

\begin{proof} 
	The asymptotic formula is from Chernaya~\cite{Ch95}, see also 
	Gruber~\cite{Gr04}. 
	Also the lower bound on $\xi_B$ is contained in \cite{Ch95}. 
	To prove the upper bound on $\xi_B$ we compare with \eqref{asy1} and take the case 
	$\lambda^d(D_d) =1$.
	We obtain $\xi_B \le \Theta_B^{1/d}$, but even more is true: 
	Assume that an information mapping is given and fixed, 
	$N_n(f) = (f(x_1), \dots , f(x_n))$. Then we can define the radius of information 
	for the two problems $\APP_\infty$ and $\INT$ by
	$
r(N_n, \APP_\infty) = \sup_{\Vert f \Vert \le 1, \, N_n(f)=0} \Vert f \Vert_\infty
$
and
		$
r(N_n, \INT ) = \sup_{\Vert f \Vert \le 1, \, N_n(f)=0} \int_{D_d} f(x) \, {\rm d} x 
$
 and obtain 
	$$
	r(N_n, \INT) \le \frac{d}{d+1} \,  r (N_n, \APP_\infty) 
	$$
	and the statement follows. 

Similarly, we obtain the  bound
	$$ 
		\inf_{D_d}   e_n (F^B(D_d), \INT, \lstd  ) 
		\cdot \lambda^d(D_d)^{-(d+1)/d}   \le   \frac{d}{d+1}  \lambda^d(B)^{-1/d} 
		\cdot n^{-1/d} . 
$$	
To prove equality it is again enough to consider $D_d$ as the disjoint union 
	of $n$ balls with equal radius.
\end{proof} 

\begin{rem}
	The results for $L_1$-approximation (or integration) 
	and $L_\infty$-approximation are very similar. For normed problems with $\lambda^d(D_d) =1$ 
	the optimal error bounds differ, \change{asymptotically for large $n$,  
	at most by a factor 
	$$
K_d \in [ 1, \frac{d+1}{d} ( d \log d + d \log \log d + 5d)^{1/d}],
	$$
	which is close to one. }
	Therefore we do not study $L_p$-approximation for $1 < p < \infty$ in detail. 
\end{rem} 

	Again one may use \eqref{lower4} to prove the curse of dimension. 
	We formulate the result as a corollary. 
	\change{ The case $p=2$ is already contained in~\cite{HNUW14b}, 
	see Proposition~3.2. 
If one puts $a= 1/(2 \e) $ then one obtains 
$n(\e , d) \ge \frac{1}{2} ( \sqrt{72 e \pi } \e )^{-d} $,  
and similar results are contained in
	Sukharev~\cite{Su79}.}  
	
\begin{cor}  \label{cor1} 
	Assume that $D_d \subset \R^d$ is Jordan measurable 
	with  $\lambda^d(D_d) = 1  $ and consider functions on $D_d$ with 
$$
	|f(x) - f(y)| \le  d^{-1/p}  \Vert x - y \Vert_p.
$$  
	Then we need, for the integration problem 
	$$
	S_d(f) = \int_{D_d} f(x) \, {\rm d} x ,
	$$
at least \change{ $n(\e , d) \ge (\tilde c_p / \e)^d$ } 
function values 
	to reach the error $\eps$, i.e., the problem suffers from the curse of dimension.
\end{cor}

There are two directions to continue these studies:
\begin{itemize}
\item
We may  study other function spaces. 
\item
We may allow general linear information instead of 
function evaluation.
\end{itemize}
We will follow both directions in the following. 

\section{Other function spaces} 

We start with a result from \cite{HNUW17}.
For any open set $D_d \subset \R^d$ with volume $\lambda^d(D_d)=1$ 
we consider the set 
$$
C^r(D_d) = \{ f\: D_d \to \R \mid \Vert D^\beta f \Vert_\infty \le 1, \ 
|\beta |_1 \le r \}. 
$$

\begin{thm}  \label{XX} 
For all $r \in \N$ there exists a constant $c_r>0$ such that for all 
$d,n \in \N$ 
$$
e_n (C^r(D_d), \INT, \lstd ) \ge \min \{ 1/2, \,  c_r \, d \, n^{-r/d} \} . 
$$ 
\end{thm}

Observe that the constant $c_r$ does not depend on $D_d$ or $d$, we have an explicit 
uniform lower bound; the same lower bound holds for the infimum 
over all $D_d$ (with volume 1). 

\begin{rem} 
1) The lower bound cannot be improved since for cubes we have a similar upper bound. 
It would be good to know more on the constants $c_r$ and on extremal sets $D_d$,
where $e_n(C^r(D_d), \INT, \lstd )$ is small,  for given $r$, $d$ and $n$. 

For these function spaces there cannot be a meaningful explicit uniform  upper bound 
since
$\sup_{D_d}  e_n (C^r(D_d), \INT, \lstd ) =1$. 
The supremum over $D_d$ makes sense if we impose boundary conditions 
such as $f(x) = 0$ for $x \in \partial D$.

2) 
It is known that 
the weak order
$$
e_n (C^r(D_d), \INT, \lstd  ) \asymp n^{-r/d} 
$$
holds at least for every bounded Lipschitz domain. 
This follows from much more general results of \cite{NWW04,NT06}.  
We guess that more is true and the asymptotic constant 
$$
\lim_{n \to \infty} e_n(C^r(D_d), \INT, \lstd ) \cdot n^{r/d} = C_{D_d} 
$$
does not depend on $D_d$ (for fixed $d$) if $D_d$ is Jordan measurable
with $\lambda^d(D_d)=1$.  
\end{rem}

\begin{rem}
The problem $\APP_\infty$ for the same spaces and norms was studied 
by Krieg~\cite{Kr19}. The lower bound 
$d\, n^{-r/d}$ is now replaced by $d^{r/2} \, n^{-r/d}$ 
if $r$ is even and again this cannot be  improved since the bound 
is sharp for the cube. 
If $r$ is odd and $r \ge 3$ then the exact order is unknown,
Krieg proved for the cube $D_d=[0,1]^d$  the lower bound 
$
 d^{r/2} \, n^{-r/d}
$
and the upper bound  
$ d^{(r+1)/2} \,  n^{-r/d}$ . 
It follows that  
approximation is essentially more difficult then integration 
iff $r \ge 3$.
The lower bound holds for $\eps < \eps_r$ where $\eps_r$ is rather small 
and hence the case of large $\eps$ remains open.  
\end{rem} 

\begin{rem}
The norm 
$$
\max_{|\beta|_1 \le r} \Vert D^\beta f \Vert_\infty 
$$
might be reasonable when we consider a cube $D_d = [0,1]^d$
but it is not invariant with respect to rotation. 
The function $f(x) = \sum x_i$ has a gradient with length 
$d^{1/2}$ as $g(x) = d^{1/2} \cdot x_1$, but all 
partial derivatives of $f$ are bounded by one. 
Therefore we also consider an orthogonal invariant norm. 
Since for this modified space $\wt C^r(D_d)$ 
many problems are still open, we only discuss the case $r=2$ 
in the following.
\end{rem} 

\begin{example} 
Let us discuss the integration problem for the class 
$$
\wt C^2 (D_d) = \{ f\in C^2 (D_d) \mid 
\Vert f \Vert_\infty \le 1, \, 
\Lip (f) \le d^{-1/2}, \, 
\Lip (D^\Theta f ) \le d^{-1} \}, 
$$
as in \cite{HNUW14a,HNUW14b,HPU19}.
Here $D^\Theta f$ denotes any directional derivative in a direction 
$\Theta \in S^{d-1}$ and 
$$
\Lip (g) = \sup_{x,y\in D_d} \frac{|g(x)-g(y)|}{\Vert x-y\Vert_2} .
$$ 
	Again we assume that $D_d \subset \R^d$ is an \change{open set} 
	with $\lambda^d(D_d) = 1$. 
We conjecture that there exists a constant $C>0$ (independent on $d$ and $D_d$) such that 
\begin{equation}  \label{123} 
e_n( \wt C^2 (D_d) , \INT, \lstd) \ge C \cdot n^{-2/d} .
\end{equation} 

Some comments are in order: 

1) This would be another explicit uniform lower bound and, because of known upper bounds 
for the cube, it certainly cannot be improved. 


2) It is known that the integration problem for $\wt C^2 (D_d)$ and 
\emph{certain} $D_d$ suffers from the curse 
of dimensionality. This is true if $D_d$ has a small radius, see \cite{HPU19} for the 
best known results. 
For example, the curse is known if $D_d$ is a $\ell_p^d$ ball and $p \ge 2$. 
It is not known for $p$ balls and $p<2$.  

3) It is easy to see that the lower bound \eqref{123} is true if 
$D_d$ is a disjoint union of $n$ euclidean balls of the same size. 
But, of course, this domain is not extremal for the given norm. 

4) Assume that the $x_i$ form a grid
\change{and $D_d$ is a cube}. Then, for $d=1$, one may take a 
quadratic spline $f_1$ as a fooling function and for $d>1$ one can take a 
fooling function of the form 
$$
f_d(x) = \frac{1}{d} \sum_{i=1}^d f_1 (x^i) ,
$$
where $x= (x^1, \dots , x^d)$. 
	\change{Hence we obtain exactly the lower bound \eqref{123} in this case. 
	Loosely speaking, 
	the conjecture says that arbitrary $D_d$ and function values 
	at arbitrary points $x_1, x_2, \dots , x_n$  do not lead to much better 
	error bounds than grids for cubes. 
	See Theorem~\ref{XX} 
	and
	\cite{HNUW17}
	for such a result 
	for slightly 
	different norms.}

	\end{example}

\begin{rem} 
	Optimal recovery for  $C^2$ functions on general 
	domains was also studied in \cite{BBS10}. 
	The authors use function values and values of the gradient 
	as information and prove asymptotic results that, again, only depend on 
	the size of $D_d$.
\end{rem}

\begin{rem} 
Consider the compact embedding of $W^r_p(D_d)$  into $L_q(D_d)$
for a bounded Lipschitz domain $D_d$.  
For $\lall$ it is known that the rate of convergence 
for the approximation numbers (error of optimal linear algorithms) 
and the Gelfand numbers (up to a factor two the error 
	of optimal algorithms) do not depend on $D_d$, \change{ see, e.g., 
	\cite{NT06}. } 
The same is known 
for $\lstd$ and then the optimal order 
is the same for linear and nonlinear algorithms. 
The optimal order is 
$$
e_n(W^r_p(D_d), \APP_q , \lstd) \asymp n^{-r/d + (1/p-1/q)_+} 
$$
for all bounded Lipschitz domains, 
see \cite{NWW04,NT06}, 
whenever the Sobolev space is embedded into $C$.
There are many common equivalent norms for the Sobolev space, 
we may take 
$$
\Vert f \Vert_{W^r_p(D_d)}^p = \sum_{|\alpha|_1  \le r } \Vert D^\alpha f \Vert_p^p .
$$
Open Problem: Is the asymptotic constant independent on the shape of $D_d$ 
and only depends on the volume of $D_d$? 

We may ask the same question for $\lall$ and also may distinguish between all algorithms and 
the class of linear algorithms. 
See the next section for the case $p=q=2$. 
\end{rem} 

\section{Arbitrary linear information} 

Weyl~\cite{Weyl}
proved that the asymptotic constant for the size of the eigenvalues 
of the Dirichlet Laplacian 
and also of the Neumann Laplacian do not depend on the shape of $D_d$,
it only depends on the volume of $D_d$; see \cite{St08} for an accessible 
(idea of the) proof.
\change{
For the study of the Dirichlet Laplacian one can assume 
that $D_d$ is an arbitrary bounded open set $D_d \subset \R^d$; 
for the study of the Neumann Laplacian one needs stronger assumptions, usually one assumes 
that $D_d$ is a bounded Lipschitz domain. } 
 
The results of Weyl were extended by many authors, see the surveys
by
Birman and Solomjak~\cite{BS79,BS80}. 
The papers by 
Birman and Solomjak~\cite{BS70} 
and 
Tulovsky~\cite{Tu72}
are important for the asymptotic constant of more general differential 
equations and boundary value problems. 
These results can be used, as explained in  
Mieth~\cite{Mi19}, 
to compute the asymptotic constant for Sobolev embeddings in the 
Hilbert space case, i.e., $p=q=2$. 

We consider the numbers 
$
e_n (H^r(D_d), \APP_2, \lall)
$, 
i.e., the approximation numbers of Sobolev embeddings. 
In addition to $H^r(D_d)$ we also consider the subspace $H_0^r(D_d)$, 
the closure of $C_0^\infty (D_d)$.  
We always assume that $D_d$ is a bounded (nonempty) \change{open set} 
in $\R^d$
and for the results concerning $H_0^r(D_d)$ this assumption is enough. 
When the whole space $H^r(D_d)$ is considered then one needs the extension 
property of $D_d$. \change{We may assume 
that $D_d$ is a bounded Lipschitz 
domain, but also more general domains are possible, see 
\cite{Ro06}}. 

We collect some results, most of them as in Mieth~\cite{Mi19},  
and add a little bit using results of Birman and Solomjak~\cite{BS70}. 

\begin{thm}  \label{thm4} 
	The asymptotic constant 
	$$ 
	\lim_{n \to \infty} e_n (H^r(D_d) , \APP_2, \lall) \cdot n^{r/d} \cdot \lambda (D_d)^{-r/d} 
	= C_{r,d} 
$$
exists and is independent of $D_d$ and also coincides with the asymptotic 
constant 
	$$ 
	\lim_{n \to \infty} e_n (H^r_0 (D_d) , \APP_2, \lall) \cdot n^{r/d} \cdot \lambda (D_d)^{-r/d} 
	= C_{r,d} 
$$
for the subspace with zero boundary values.  
\end{thm}

Hence the asymptotic constant $C_{r,d}$ does not depend on the boundary conditions 
and does not depend on the shape of $D_d$.
The norm in $H^r(D_d)$ can be 
given by $\Vert f \Vert^2 = \Vert f \Vert_2^2 + \sum_{|\alpha |_1 = r} \Vert D^\alpha f  
\Vert_2^2$ (or similar)   
and there are explicit formulas for $C_{r,d}$, see \cite{Mi19}. 

\begin{rem} 
	Theorem~\ref{thm4} is for Sobolev embeddings in the Hilbert space case.
	Here the error bounds coincide with approximation numbers or singular values. 
	It would be interesting to have similar results for embeddings 
	of $W^r_p(D_d)$ into $L_q(D_d)$ and general $p$ and $q$.  
\end{rem}

\begin{rem} 
Mieth \cite{Mi19} used results of 
Kr\"oger \cite{Kr92} and  
Li and Yau~\cite{LY83}
to prove explicit uniform upper bounds for 
$ e_n (H^r_0 (D_d) , \APP_2, \lall)$ and lower bounds for 
$ e_n (H^r (D_d) , \APP_2, \lall)$.
An extended \emph{Polya conjecture} for this case reads
	$$ 
e_n (H^r(D_d) , \APP_2, \lall) \cdot n^{r/d} \cdot \lambda (D_d)^{-r/d} 
\ge  C_{r,d} 
$$
and
	$$ 
e_n (H^r_0 (D_d) , \APP_2, \lall) \cdot n^{r/d} \cdot \lambda (D_d)^{-r/d} 
\le C_{r,d} 
$$
for all $n\in \N$. 
The known results, see Mieth~\cite{Mi19}, are only slightly weaker than 
these conjectured ones. 
\end{rem}

We finish the paper with a remark on $C^\infty$ functions.  

\begin{rem} 
Functions from the class $C^\infty$
with the norm $\Vert f \Vert := \sup_{\alpha \in \N_0^d} \Vert D^\alpha f \Vert_\infty$ 
are studied in \cite{NW09,We12}, see also Vyb\'\i ral~\cite{Vy14}. 
The curse is proved for domains 
$D_d$ of the form $D_d = [a,b]^d$, where $b-a>0$ can be small 
but is independent of $d$. 
By the proof technique it is clear that all proved lower bounds also hold for larger 
domains. Nevertheless, the proof does not cover all $D_d$ with a size 
$\lambda^d (D_d) \ge \alpha^d$ and it would be interesting to know whether the curse 
also holds for these more general domains.

If one assumes that all directional derivatives of all orders are bounded by 
one and $D_d = [0,1]^d$ then one can prove the weak tractability of the
integration problem using the 
Clenshaw-Curtis Smolyak algorithm, see \cite{HNU14}.
\change{
This result was extended by G.~Xu~\cite{Xu} who proved weak tractability for the same 
class and $L_q$-approximation for $q<\infty$. The used algorithm is a Smolyak algorithm 
introduced in \cite{BNR}.
The case of $L_\infty$-approximation is still open, hence this is another unsolved problem. 
} 
\end{rem}

\medskip

{\bf Acknowledgment.}  \ 
I thank several friends, colleagues and referees for valuable comments.

\end{document}